\documentclass[review,onefignum,onetabnum]{siamart171218}
\usepackage[left = 1.7cm, right = 1.7cm, top = 2cm, bottom = 2cm]{geometry}
\usepackage{tikz}
\usepackage{xstring}
\usepackage{tkz-euclide}
\usepackage{tikz-cd}
\usepackage[normalem]{ulem}
\usepackage{amssymb}
\usepackage{mathtools}
\usepackage{cancel}
\usepackage{subfigure}
\usepackage{subcaption}
\usepackage{mathrsfs}
\definecolor{codegreen}{rgb}{0,0.6,0}
\usetikzlibrary{patterns.meta}
\usetikzlibrary{decorations.pathreplacing,calligraphy}

\usepackage{lipsum}
\usepackage{amsfonts}
\usepackage{graphicx}
\usepackage{tikz-cd}
\usepackage{epstopdf}
\usepackage{algorithmic}
\ifpdf
  \DeclareGraphicsExtensions{.eps,.pdf,.png,.jpg}
\else
  \DeclareGraphicsExtensions{.eps}
\fi

\newsiamremark{remark}{Remark}
\newsiamremark{hypothesis}{Hypothesis}
\crefname{hypothesis}{Hypothesis}{Hypotheses}
\newsiamthm{claim}{Claim}

\headers{An Approximate Bounded Cochain Projection}{Gerritsma and Shrestha}

\title{An Approximate Bounded Cochain Projection}

\author{Marc Gerritsma\thanks{TU Delft, The Netherlands 
  (\email{m.i.gerritsma@tudelft.nl}).}
\and Suyash Shrestha\thanks{TU Delft, The Netherlands 
  (\email{s.shrestha-1@tudelft.nl}).}}

\usepackage{amsopn}

\makeatletter
\newcommand*{\addFileDependency}[1]{
  \typeout{(#1)}
  \@addtofilelist{#1}
  \IfFileExists{#1}{}{\typeout{No file #1.}}
}
\makeatother

\newcommand*{\myexternaldocument}[1]{%
    \externaldocument{#1}%
    \addFileDependency{#1.tex}%
    \addFileDependency{#1.aux}%
}

\ifpdf
\hypersetup{
  pdftitle={An Approximate Bounded Cochain Projection},
  pdfauthor={Marc Gerritsma and Suyash Shrestha}
}
\fi

\myexternaldocument{ex_supplement}

\newcommand{\bilnear}[2]{(#1, #2)}

\newcommand{\kForm}[2]{#1^{(#2)}}

\newcommand{\normHk}[2]{\| #1 \|_{H\Lambda^{(#2)}(\Omega)}}

\newcommand{\normLtwo}[2]{\| #1 \|_{L^2 \Lambda^{(#2)}(\Omega)}}

\begin{document}

\maketitle

\begin{abstract}
  This paper presents a construction of a projector from an infinite-dimensional Hilbert complex of differential $k$-forms onto a finite-dimensional piecewise polynomial sub-complex. We demonstrate that, on contractable domains, the proposed projector attains the three properties of the Bounded Cochain Projector, namely that the projector is idempotent, uniformly bounded in the Sobolev norm, and it commutes with the exterior derivative. On non-contractible domains, the projector remains idempotent and uniformly bounded, while the commuting property is satisfied up to arbitrary accuracy.
\end{abstract}



\section{Introduction}
Let $\Omega$ be an $n$-dimensional manifold with sufficiently smooth Lipschitz boundary. We consider the closed Hilbert complex of differential forms introduced in \cite{AFW06},
\[
\mathbb{R}
\hookrightarrow
H\Lambda^{(0)}(\Omega)
\xrightarrow{\;\mathrm{d}\;}
H\Lambda^{(1)}(\Omega)
\xrightarrow{\;\mathrm{d}\;}
\cdots
\xrightarrow{\;\mathrm{d}\;}
H\Lambda^{(n-1)}(\Omega)
\xrightarrow{\;\mathrm{d}\;}
H\Lambda^{(n)}(\Omega)
\longrightarrow 0.
\]
The exactness of this sequence depends on the global topology of $\Omega$. In particular, if $\Omega$ is contractible, the de~Rham complex is exact, meaning that the range of the exterior derivative equals the null space of this operator in the subsequent space. In contrast, when $\Omega$ is not contractible, the complex possesses a nontrivial cohomology, and the range of the exterior derivative forms only a subset of the null space of the subsequent exterior derivative.

For each space $H\Lambda^{(k)}(\Omega)$, we introduce a finite-dimensional subspace
\(
H\Lambda_h^{(k)}(\Omega)\subset H\Lambda^{(k)}(\Omega),
\)
chosen such that the exterior derivative maps between successive finite-dimensional function spaces. This yields a finite-dimensional subcomplex
\[
\mathbb{R}
\hookrightarrow
H\Lambda_h^{(0)}(\Omega)
\xrightarrow{\;\mathrm{d}\;}
H\Lambda_h^{(1)}(\Omega)
\xrightarrow{\;\mathrm{d}\;}
\cdots
\xrightarrow{\;\mathrm{d}\;}
H\Lambda_h^{(n-1)}(\Omega)
\xrightarrow{\;\mathrm{d}\;}
H\Lambda_h^{(n)}(\Omega)
\longrightarrow 0.
\]



A bounded cochain projection, $\Pi^{(k)}$, maps elements of $H\Lambda^{(k)}(\Omega)$ onto $H\Lambda_h^{(k)}(\Omega)$ such that the projector is idempotent, i.e. $\Pi^{(k)} \Pi^{(k)} = \Pi^{(k)}$, the projected solution is bounded in the Sobolev norm, $\normHk{ \Pi^{(k)}\phi^{(k)}}{k} \leq C \normHk{ \phi^{(k)}}{k}$ for all $\phi^{(k)}\in H\Lambda^{(k)}(\Omega)$, and the projector commutes with the exterior derivative, $\mathrm{d} \Pi^{(k)} = \Pi^{(k+1)} \mathrm{d}$ leading to the following commuting diagram 

\[
\begin{tikzcd}
    \hdots \arrow[r] & H\Lambda^{(k - 1)}(\Omega) \arrow[r,"\mathrm{d}"] \arrow[d,"\Pi^{(k - 1)}"] & H\Lambda^{(k)}(\Omega) \arrow[r,"\mathrm{d}"] \arrow[d,"\Pi^{(k)}"] & H\Lambda^{(k + 1)}(\Omega) \arrow[r] \arrow[d,"\Pi^{(k + 1)}"] & \hdots \\
     \hdots \arrow[r] & H\Lambda_h^{(k - 1)}(\Omega) \arrow[r,"\mathrm{d}"] & H\Lambda_h^{(k)}(\Omega) \arrow[r,"\mathrm{d}"] & H\Lambda_h^{(k + 1)}(\Omega) \arrow[r] & \hdots
\end{tikzcd}
\]

In \cite{AFW06, AFW10} the bounded cochain projection was introduced. The construction of local bounded cochain projections is presented in \cite{FW12} and is extended to double complexes in \cite{FW14}. Subsequent work, \cite{Arnold2021LocalFEEC, Chaumont-Frelet2024AvarvecHtextrmcurl}, refined these ideas in several directions, including local constructions, improved approximation properties, preservation of boundary conditions, and extensions to general finite element systems. More recently, \cite{Ern2025LocalSplits} developed local $L^2$-bounded commuting projections by replacing the continuous regularisation used in previous constructions, such as \cite{ChristiansenWinther06}, with local discrete problems posed on an auxiliary refinement. Their construction establishes fully local commuting projections with provable local $L^2$ stability while avoiding the analytical difficulties associated with continuous right inverses of the exterior derivative. This represents the first completely discrete construction of locally $L^2$-bounded commuting projections. Likewise, in \cite{PintoSchnack23}, commuting $L^2$-projections on non-matching interfaces are presented. These papers employ finite-dimensional function spaces using Finite Element Exterior Calculus (FEEC). Similar results are obtained with the Mimetic Spectral Element Method, \cite{Jain2021ConstructionMeshes}. The only requirement on the finite-dimensional subspaces is that they form a de Rham sequence, i.e. $\mathrm{d}H\Lambda_h^{(k)}(\Omega) \subset H\Lambda_h^{(k+1)}(\Omega)$. 

The goal of the present work is to formulate a construction of a projector that fulfils the three aforementioned properties of a bounded cochian projection. 
The paper is organised as follows. In Section~\ref{sec:prerequisites} some notation is introduced followed by the introduction of the canonical projector in Section~\ref{sec:cano_proj}. In Section~\ref{sec:Augmented} a mixed formulation is presented for which the existence and uniqueness of a solution will be shown. The properties of the presented mixed formulation are analysed in Section~\ref{sec:BCP}. Lastly, the construction of the proposed projector based on a two-level discrete complex is presented in Section~\ref{sec:pro_proj}.

\section{Prerequisites}\label{sec:prerequisites}
In the de~Rham complex of an $n$-manifold $\Omega$ with sufficiently smooth Lipschitz boundary, we equip the spaces of differential forms with the standard $L^2$ and graph norms. The $L^2$ inner-product on $H\Lambda^{(k)}(\Omega)$ is defined by
\[
(\kForm{\alpha}{k},\kForm{\beta}{k})
:=
\int_{\Omega}
\kForm{\alpha}{k}
\wedge
\star \kForm{\beta}{k},
\qquad
\forall\,
\kForm{\alpha}{k},
\kForm{\beta}{k}
\in
H\Lambda^{(k)}(\Omega).
\]
The associated $L^2$ norm is given by
\[
\normLtwo{\kForm{\alpha}{k}}{k}^2
=
(\kForm{\alpha}{k},\kForm{\alpha}{k}).
\]
Following \cite{AFW06}, we define the $H\Lambda^{(k)}$ inner-product by
\[
(\kForm{\alpha}{k},\kForm{\beta}{k})_{H\kForm{\Lambda}{k}(\Omega)}
:=
\int_{\Omega}
\kForm{\alpha}{k}
\wedge
\star \kForm{\beta}{k}
+
\int_{\Omega}
\mathrm{d}\kForm{\alpha}{k}
\wedge
\star \mathrm{d}\kForm{\beta}{k} =
(\kForm{\alpha}{k},\kForm{\beta}{k})
+
(\mathrm{d}\kForm{\alpha}{k},
 \mathrm{d}\kForm{\beta}{k}),
\qquad
\forall\,
\kForm{\alpha}{k},
\kForm{\beta}{k}
\in
H\Lambda^{(k)}(\Omega).
\]
The corresponding $H$-norm is
\[
\normHk{\kForm{\omega}{k}}{k}^2
=
(\kForm{\omega}{k},\kForm{\omega}{k})_{H\kForm{\Lambda}{k}(\Omega)}
=
\normLtwo{\kForm{\omega}{k}}{k}^2
+
\normLtwo{\mathrm{d}\kForm{\omega}{k}}{k+1}^2.
\]

We denote the space of closed differential forms by
\[
\mathcal{Z}^{(k)}
=
\left\{
\kForm{\omega}{k}
\in
H\Lambda^{(k)}(\Omega)
\;\big|\;
\mathrm{d}\kForm{\omega}{k}=0
\right\},
\]
and its $L^2$-orthogonal complement by
\[
\mathcal{Z}^{(k),\perp}
=
\left\{
\kForm{\eta}{k}
\in
H\Lambda^{(k)}(\Omega)
\;\big|\;
(\kForm{\eta}{k},\kForm{\omega}{k})=0,
\quad
\forall\,
\kForm{\omega}{k}
\in
\mathcal{Z}^{(k)}
\right\}.
\]
For every
\(
\kForm{\omega}{k}
\in
\mathcal{Z}^{(k),\perp}
\),
the Poincar\'e inequality holds:
\begin{equation}
\normHk{\kForm{\omega}{k}}{k}
\leq
C_P
\normLtwo{\mathrm{d}\kForm{\omega}{k}}{k+1}.
\label{eq:Poincare}
\end{equation}
Since
\(
H\Lambda^{(k)}(\Omega)
=
\mathcal{Z}^{(k)}
\oplus
\mathcal{Z}^{(k),\perp},
\)
every
\(
\kForm{\omega}{k}
\in
H\Lambda^{(k)}(\Omega)
\)
admits a unique decomposition
\(
\kForm{\omega}{k}
=
\kForm{\omega}{k}_0
+
\kForm{\omega}{k}_{\perp},
\)
where
\(
\kForm{\omega}{k}_0
\in
\mathcal{Z}^{(k)}
\)
and
\(
\kForm{\omega}{k}_{\perp}
\in
\mathcal{Z}^{(k),\perp}
\).

We define the space of exact differential $k$-forms by
\[
\mathcal{R}^{(k)}
=
\left\{
\kForm{\omega}{k}
\;\big|\;
\exists\,
\kForm{\alpha}{k-1}
\in
H\Lambda^{(k-1)}(\Omega)
\;\text{s.t.}\;
\kForm{\omega}{k}
=
\mathrm{d}\kForm{\alpha}{k-1}
\right\},
\]
and the space of harmonic forms by
\[
\mathcal{H}^{(k)}
=
\left\{
\kForm{\zeta}{k}
\in
\mathcal{Z}^{(k)}
\;\big|\;
\cancel{\exists}
\,
\kForm{\beta}{k-1}
\in
H\kForm{\Lambda}{k-1}(\Omega)
\;\text{s.t.}\;
\kForm{\zeta}{k}
=
\mathrm{d}\kForm{\beta}{k-1}
\right\}.
\]
Since the exterior derivative is nilpotent we have $\mathcal{R}^{(k)} \subseteq \mathcal{Z}^{(k)}$. More precisely, we have the decomposition $\mathcal{Z}^{(k)} = \mathcal{R}^{(k)} \oplus \mathcal{H}^{(k)}$ on general non-contractable domains. On contractible domains, $\mathcal{H}^{(k)} = \{0\}$, so every closed form is exact.

Let $H\Lambda_h^{(k)}(\Omega) \subset H\Lambda^{(k)}(\Omega)$ be a conforming finite-dimensional subspace. We define the space of closed forms by
\[
\mathcal{Z}_h^{(k)}
=
\left\{
\kForm{\omega}{k}
\in
H\Lambda_h^{(k)}(\Omega)
\;\big|\;
\mathrm{d}\kForm{\omega}{k}=0
\right\},
\]
and its orthogonal complement by
\[
\mathcal{Z}_h^{(k),\perp}
=
\left\{
\kForm{\eta}{k}
\in
H\Lambda_h^{(k)}(\Omega)
\;\big|\;
(\kForm{\eta}{k},\kForm{\omega}{k})=0,
\quad
\forall\,
\kForm{\omega}{k}
\in
\mathcal{Z}_h^{(k)}
\right\}.
\]
Consequently, we have the decomposition of $H\kForm{\Lambda_h}{k}(\Omega)$
\[
H\Lambda_h^{(k)}(\Omega)
=
\mathcal{Z}_h^{(k)}
\oplus
\mathcal{Z}_h^{(k),\perp}.
\]
The finite-dimensional spaces of exact and harmonic differential forms are defined as
\[
\mathcal{R}_h^{(k)}
=
\left\{
\kForm{\omega}{k}
\;\big|\;
\exists\,
\kForm{\alpha}{k-1}
\in
H\Lambda_h^{(k-1)}(\Omega)
\;\text{s.t.}\;
\kForm{\omega}{k}
=
\mathrm{d}\kForm{\alpha}{k-1}
\right\},
\]
and
\[
\mathcal{H}_h^{(k)}
=
\left\{
\kForm{\zeta}{k}
\in
\mathcal{Z}_h^{(k)}
\;\big|\;
\cancel{\exists}
\,
\kForm{\beta}{k-1}
\in
H\kForm{\Lambda_h}{k-1}(\Omega)
\;\text{s.t.}\;
\kForm{\zeta}{k}
=
\mathrm{d}\kForm{\beta}{k-1}
\right\},
\]
which yields the decomposition of $\kForm{\mathcal{Z}_h}{k}$
\[
\mathcal{Z}_h^{(k)}
=
\mathcal{R}_h^{(k)}
\oplus
\mathcal{H}_h^{(k)}.
\]

In general, the harmonic space of $H\kForm{\Lambda_h}{k}(\Omega)$ is not a subspace of the harmonic space in $H\kForm{\Lambda}{k}(\Omega)$. To quantify the discrepancy between these spaces, \cite{AFW10} employs the notion of the gap between closed subspaces of a Hilbert space \cite[Chapter IV, \S 2.1]{Kato_1995}. For two subspaces
\(
X,Y \subset H\Lambda^{(k)}(\Omega)
\),
define
\[
\delta(X,Y)
=
\sup_{\substack{x\in X\\x\neq 0}}
\inf_{y\in Y}
\frac{\normHk{x-y}{k}}
     {\normHk{x}{k}},
\qquad
\mathrm{gap}(X,Y)
=
\max\{\delta(X,Y),\delta(Y,X)\}.
\]
The distance between the infinite and finite-dimensional harmonic spaces is therefore measured as
\[
\mathrm{gap}
\bigl(
\mathcal{H}^{(k)},
\mathcal{H}_h^{(k)}
\bigr).
\]
A fundamental result of FEEC \cite{AFW10} states that if the finite-dimensional spaces form a bounded subcomplex, then
\[
\mathrm{gap}
\bigl(
\mathcal{H}^{(k)},
\mathcal{H}_h^{(k)}
\bigr)
\longrightarrow 0,
\qquad
h\longrightarrow 0.
\]
This guarantees convergence of finite-dimensional harmonic forms to their infinite-dimensional counterparts. In this work, we shall use the terminology \emph{topologically conforming} to refer to case where $\mathrm{gap}
\bigl(
\mathcal{H}^{(k)},
\mathcal{H}_h^{(k)}
\bigr) = 0$.

\section{The canonical projector}
\label{sec:cano_proj}

Let $H\Lambda_h^{(k)}(\Omega) \subset H\Lambda^{(k)}(\Omega)$ be a conforming finite element subcomplex. A central object in FEEC is a sequence of projection operators, \cite{Ern2025LocalSplits,FW12,framework}
\[
\Pi^{(k)}_{\mathscr{C}} : H\Lambda^{(k)}(\Omega) \to H\Lambda_h^{(k)}(\Omega),
\]
collectively referred to as the canonical (cochain) projector. The linear projector $\Pi^{(k)}_{\mathscr{C}}$ is idempotent and it satisfies the commuting property, i.e.
\[
\Pi^{(k)}_{\mathscr{C}} \Pi^{(k)}_{\mathscr{C}} = \Pi^{(k)}_{\mathscr{C}}, \qquad \mathrm{d}  \Pi^{(k)}_{\mathscr{C}}
=
\Pi^{(k+1)}_{\mathscr{C}} \mathrm{d},
\]
\[
\begin{tikzcd}
    \hdots \arrow[r] &
    H\Lambda^{(k-1)}(\Omega)
    \arrow[r, "\mathrm{d}"]
    \arrow[d, "\Pi^{(k-1)}_{\mathscr{C}}"]
    &
    H\Lambda^{(k)}(\Omega)
    \arrow[r, "\mathrm{d}"]
    \arrow[d, "\Pi^{(k)}_{\mathscr{C}}"]
    &
    H\Lambda^{(k+1)}(\Omega)
    \arrow[r, "\mathrm{d}"]
    \arrow[d, "\Pi^{(k+1)}_{\mathscr{C}}"]
    &
    \hdots
    \\
    \hdots \arrow[r] &
    H\Lambda_h^{(k-1)}(\Omega)
    \arrow[r, "\mathrm{d}"]
    &
    H\Lambda_h^{(k)}(\Omega)
    \arrow[r, "\mathrm{d}"]
    &
    H\Lambda_h^{(k+1)}(\Omega)
    \arrow[r]
    &
    \hdots
\end{tikzcd}
\]

Generally, the elements of $H\kForm{\Lambda}{k}(\Omega)$ lack the regularity\footnote{In fact, the elements of $H\kForm{\Lambda}{k}(\Omega)$ are equivalence classes of differential forms and not differential forms themselves.} to directly apply the canonical projector. Typically, $\Pi^{(k)}_{\mathscr{C}}$ is constructed as a composition of a smoothing operator followed by a local interpolation operator. A typical representation is
\[
\Pi^{(k)}_{\mathscr{C}} = I_h^{(k)}  \mathcal{S}^{(k)},
\]
where $\mathcal{S}^{(k)} : H\Lambda^{(k)}(\Omega) \to C^\infty \Lambda^{(k)}(\Omega)$ is a regularization (or smoothing) operator ensuring sufficient pointwise regularity, and
\[
I_h^{(k)} : C^\infty \Lambda^{(k)}(\Omega) \to H\Lambda_h^{(k)}(\Omega)
\]
is a canonical interpolation operator defined by the degrees of freedom of the finite element space. The commuting property is inherited from the compatibility of the smoothing operator with the exterior derivative and the design of the finite element degrees of freedom.

However, $\Pi^{(k)}_{\mathscr{C}}$ is, in general, not uniformly bounded in the $H\Lambda^{(k)}$ norm. That is, there does not necessarily exist a constant $C > 0$, independent of the mesh parameter $h$, such that
\[
\normHk{\Pi^{(k)}_{\mathscr{C}} \kForm{\omega}{k}}{k}
\leq
C \, \normHk{\kForm{\omega}{k}}{k},
\qquad
\forall \kForm{\omega}{k} \in H\Lambda^{(k)}(\Omega).
\]

\section{Augmented mixed formulation} \label{sec:Augmented}
We now introduce an augmented mixed formulation, which will form a part of our proposed bounded cochain projector. In this section, we present the augmented system and prove well-posedness.

We consider the following variational formulation: Given $\kForm{\phi}{k}\in H\Lambda^{(k)}(\Omega)$, find $\lambda^{(i)} \in H\Lambda_h^{(i)}(\Omega)$ for $i=k,\ldots,k+2$ and $\kForm{h}{k+1}\in \mathcal{H}_h^{(k+1)}$, such that
$\forall \kForm{\eta}{j}\in \kForm{H\Lambda_h}{j + 1}(\Omega)$ for $j=k,\ldots, k+2$ and $\forall \kForm{\gamma}{k+1}\in \mathcal{H}_h^{(k+1)}$, we satisfy

\begin{equation}
    \left \{ \begin{array}{ccccccccc}
    \bilnear{\kForm{\eta}{k}}{\kForm{\lambda}{k}} & + & \bilnear{\mathrm{d} \kForm{\eta}{k}}{\kForm{\lambda}{k + 1}} & & & & & = & \bilnear{\kForm{\eta}{k}}{\kForm{\phi}{k}}  \\ [1.5ex]
    \bilnear{\kForm{\eta}{k + 1}}{\mathrm{d} \kForm{\lambda}{k}} &  & + &  & \bilnear{\mathrm{d} \kForm{\eta}{k + 1}}{\kForm{\lambda}{k + 2}}  & + &\bilnear{\kForm{\eta}{k + 1}}{\kForm{h}{k + 1}} & = & \bilnear{\kForm{\eta}{k + 1}}{\mathrm{d}\kForm{\phi}{k}} \\ [1.5ex]
      & & \bilnear{\kForm{\eta}{k + 2}}{\mathrm{d} \kForm{\lambda}{k + 1}} & &  & & & = & \bilnear{\kForm{\eta}{k + 2}}{0} \\ [1.5ex]
     & & \bilnear{\kForm{\gamma}{k + 1}}{\kForm{\lambda}{k + 1}} & & & & & = & \bilnear{\kForm{\gamma}{k + 1}}{0}
    \end{array} \right . \; .
\label{eq:BCP}
\end{equation}

Note that the system is symmetric. To establish the well-posedness of \eqref{eq:BCP}, we verify the conditions noted in \cite{Boffi2013MixedApplications}.

\begin{proposition}[Stability conditions]
The system \eqref{eq:BCP} is well-posed, provided that

\begin{enumerate}
    \item the bilinear form
    $\bilnear{\kForm{\eta}{k}}{\kForm{\lambda}{k}}$
    is bounded and coercive on
    $\mathcal{Z}_h^{(k)}$

    \item there exists a constant
    $\beta_{\mathcal H}>0$
    independent of the mesh size, such that
    \begin{equation}
        \inf_{\kForm{h}{k+1}\in\mathcal{H}_h^{(k+1)}}
        \sup_{\kForm{\eta}{k+1}\in H\Lambda_h^{(k+1)}(\Omega)}
        \frac{
            \bilnear{\kForm{\eta}{k+1}}{\kForm{h}{k+1}}
        }{
            \|\kForm{\eta}{k+1}\|_{H\Lambda_h^{(k+1)}(\Omega)}
            \|\kForm{h}{k+1}\|_{H\Lambda_h^{(k+1)}(\Omega)}
        }
        \geq
        \beta_{\mathcal H};
    \end{equation}

    \item the bilinear forms
    $\bilnear{\mathrm{d}\kForm{\eta}{j}}{\kForm{\lambda}{j+1}}$,
    $j=k,k+1$,
    satisfy the inf-sup condition
    \begin{equation}
        \adjustlimits\inf_{\kForm{\lambda}{j+1}\in H\Lambda_h^{(j+1)}(\Omega)}
        \sup_{\kForm{\eta}{j}\in\mathcal{Z}_h^{(j),\perp}}
        \frac{
            \bilnear{\mathrm{d}\kForm{\eta}{j}}
                    {\kForm{\lambda}{j+1}}
        }{
            \normHk{\kForm{\eta}{j}}{j}
            \normHk{\kForm{\lambda}{j+1}}{j+1}
        }
        \geq
        \beta_j
        >
        0.
        \label{eq:infsup}
    \end{equation}
\end{enumerate}
\end{proposition}

\begin{remark}
    Note that the supremum in \eqref{eq:infsup} may be restricted to
    $\mathcal{Z}_h^{(j),\perp}$,
    since the supremum will never be achieved for $\kForm{\eta}{j}\in \mathcal{Z}_h^{(j)}$.
\end{remark}

\begin{lemma}[Coercivity in $\mathcal{Z}^{(k)}$]
The bilinear form
$\bilnear{\kForm{\eta}{k}}{\kForm{\lambda}{k}}$
is bounded and coercive on
$\mathcal{Z}_h^{(k)}$
with coercivity constant equal to one.
\end{lemma}

\begin{proof}
The bilinear form is symmetric and therefore, trivially bounded. Moreover, for any
$\overline{\lambda}_0^{(k)}
\in
\mathcal{Z}_h^{(k)}$
we have
\[
\bilnear{\overline{\lambda}_0^{(k)}}{\overline{\lambda}_0^{(k)}} 
=
\normLtwo{\overline{\lambda}_0^{(k)}}{k}^2
=
\normHk{\overline{\lambda}_0^{(k)}}{k}^2 \;.
\]
Hence, the bilinear form is coercive on
$\mathcal{Z}_h^{(k)}\subset \mathcal{Z}^{(k)}$
with coercivity constant equal to one.
\end{proof}

\begin{lemma}[Inf-sup condition on the harmonic space]
The harmonic bilinear form satisfies the inf-sup condition with
$\beta_{\mathcal H}=1$.
\end{lemma}

\begin{proof}
For any
$\kForm{h}{k+1}\in\mathcal{H}_h^{(k+1)}$
we may choose
$\kForm{\eta}{k+1}=\kForm{h}{k+1}$.
Then we have
\[
\sup_{\kForm{\eta}{k+1}\in H\Lambda_h^{(k+1)}(\Omega)} \frac{\bilnear{\kForm{\eta}{k+1}}{\kForm{h}{k+1}}}{\| \kForm{\eta}{k+1} \|_{H\Lambda_h^{(k+1)}(\Omega)}}  \geq   \frac{\bilnear{\kForm{h}{k+1}}{\kForm{h}{k+1}}}{\| \kForm{h}{k+1} \|_{H\Lambda_h^{(k+1)}(\Omega)}}
 = \normHk{\kForm{h}{k+1}}{k+1} \;,
\]
where we used the fact that $\normHk{h}{k+1}=\normLtwo{h}{k+1}$. Thus the inf-sup condition holds with
$\beta_{\mathcal H}=1$.
\end{proof}

\begin{lemma}[Inf-sup condition for the exterior derivative]
The inf-sup condition \eqref{eq:infsup} holds for $j=k$ with
$\beta_k=\frac{1}{C_P}$.
\end{lemma}

\begin{proof}
The third equation of \eqref{eq:BCP} implies that
$\kForm{\lambda}{k+1}\in\mathcal{Z}_h^{(k+1)}$,
while the fourth equation implies that
$\kForm{\lambda}{k+1}$
is orthogonal to
$\mathcal{H}_h^{(k+1)}$.
We thus conclude that 
\(
\kForm{\lambda}{k+1}
\in
\mathcal{R}_h^{(k+1)}.
\)
Therefore there exists
$\kForm{\overline{\eta}}{k}
\in
H\Lambda_h^{(k)}(\Omega)$
such that
\(
\mathrm{d}\kForm{\overline{\eta}}{k}
=
\kForm{\lambda}{k+1}.
\)
Using this in the supremum gives
\[
\sup_{\kForm{\eta}{k}\in\mathcal{Z}_h^{(k),\perp}}
\frac{
\bilnear{\mathrm{d}\kForm{\eta}{k}}
        {\kForm{\lambda}{k+1}}
}{
\normHk{\kForm{\eta}{k}}{k}
}
\geq
\frac{
\bilnear{\mathrm{d}\kForm{\overline{\eta}}{k}}
        {\kForm{\lambda}{k+1}}
}{
\normHk{\kForm{\overline{\eta}}{k}}{k}
}.
\]
The Poincaré inequality tells us
\(
\normHk{\kForm{\overline{\eta}}{k}}{k}
\leq
C_P
\normLtwo{\kForm{\lambda}{k+1}}{k+1},
\)
and therefore
\[
\sup_{\kForm{\eta}{k}\in\mathcal{Z}_h^{(k),\perp}}
\frac{
\bilnear{\mathrm{d}\kForm{\eta}{k}}
        {\kForm{\lambda}{k+1}}
}{
\normHk{\kForm{\eta}{k}}{k}
}
\geq
\frac{1}{C_P}
\frac{
\normLtwo{\kForm{\lambda}{k+1}}{k+1}^2
}{
\normHk{\kForm{\lambda}{k+1}}{k+1}
} \geq \frac{1}{C_P} \normHk{ \lambda^{(k+1)} }{k+1}.
\]
where we used $\normLtwo{\lambda^{(k+1)}}{k+1}=\normHk{ \lambda^{(k+1)}}{k+1}$. So \eqref{eq:infsup} is satisfied for $j=k$.
\end{proof}

\begin{theorem}[Existence and uniqueness]\label{the:well_posed}
The system \eqref{eq:BCP} admits a solution.
The variables
$\kForm{\lambda}{k}$,
$\kForm{\lambda}{k+1}$,
and
$\kForm{h}{k+1}$
are uniquely determined, whereas
$\kForm{\lambda}{k+2}$
is determined only up to an element of
$\mathcal{R}_h^{(k+2),\perp}$.
\end{theorem}

\begin{proof}

The inf-sup relation, \eqref{eq:infsup}, is not satisfied for $j=k+1$. We can arbitrarily add elements from $\mathcal{R}_h^{(k+2),\perp}$ to $\kForm{\lambda}{k+2}$ without modifying the term $\bilnear{\mathrm{d} \kForm{\eta}{k + 1}}{\kForm{\lambda}{k + 2}}$ in \eqref{eq:BCP}.
The system \eqref{eq:BCP} is singular. For a solution to exist, the right-hand side vector must be in the range of the matrix in \eqref{eq:BCP}. Due to the symmetry of the system \eqref{eq:BCP}, we see that for $\eta^{(k)}=\eta^{(k+1)}=\gamma^{(k+1)}=0$ and $\eta^{(k+2)} \in \mathcal{R}_h^{(k+2),\perp}$ the right-hand side vector vanishes and therefore the right-hand side vector is in the range of the singular matrix.

Therefore, a solution exists. The preceding lemmas imply the uniqueness of
$\kForm{\lambda}{k}$,
$\kForm{\lambda}{k+1}$,
and
$\kForm{h}{k+1}$,
while
$\kForm{\lambda}{k+2}$
remains unique only modulo
$\mathcal{R}_h^{(k+2),\perp}$.
\end{proof}

\begin{remark}
    The singularity in \eqref{eq:BCP} can be removed by adding additional constraints and Lagrange multipliers, but that is not necessary for the purpose of this paper.
\end{remark}

\section{Properties of the augmented system}\label{sec:BCP}
Having proven the existence of a solution of the augmented system in the previous section, we now present the key properties of the system. Let
\[
\Pi^{(k)} :
H\Lambda^{(k)}(\Omega)
\rightarrow
H\Lambda_h^{(k)}(\Omega)
\]
be the linear operator defined as: For a given $\kForm{\phi}{k} \in H\kForm{\Lambda}{k}(\Omega)$, solve  \eqref{eq:BCP} and set $\Pi^{(k)} \kForm{\phi}{k} = \kForm{\lambda}{k}$. 


\begin{lemma}[Idempotence]\label{lem:idem}
The operator
\(
\Pi^{(k)} :
H\Lambda^{(k)}(\Omega)
\rightarrow
H\Lambda_h^{(k)}(\Omega)
\)
defines a projection onto $H\Lambda_h^{(k)}(\Omega)$. In particular,
\begin{equation}
    \Pi^{(k)} \Pi^{(k)} \kForm{\phi}{k} = \Pi^{(k)} \kForm{\phi}{k}, \quad \forall \kForm{\phi}{k} \in H\kForm{\Lambda}{k}(\Omega) \quad  \Longrightarrow \quad \Pi^{(k)} \kForm{\lambda}{k} = \kForm{\lambda}{k}, \quad \forall \kForm{\lambda}{k} \in H\kForm{\Lambda_h}{k}(\Omega).
\end{equation}
\end{lemma}

\begin{proof}
Let
\(
\kForm{\phi}{k}
\in
H\Lambda_h^{(k)}(\Omega)
\)
and consider the augmented system in \eqref{eq:BCP}. It is straightforward to verify that
\[
\kForm{\lambda}{k}
=
\kForm{\phi}{k},
\qquad
\kForm{\lambda}{k+1}
=
0,
\qquad
\kForm{h}{k+1}
=
0,
\qquad
\kForm{\lambda}{k+2}
=
0,
\]
satisfies all equations of \eqref{eq:BCP}. Hence,
\(
\kForm{\phi}{k}
\)
is \emph{a} solution of the problem. From Theorem~\ref{the:well_posed} we know that the solution for $\kForm{\lambda}{k}$ is unique, therefore 
it is the only solution, thus proving idempotence.
\end{proof}


\begin{lemma}[Boundedness]\label{lem:bound}
    The projection is uniformly bounded in the Sobolev norm.
\end{lemma}
\begin{proof}
    Every $\kForm{\lambda}{k}$ can be uniquely decomposed into $\kForm{\lambda}{k}=\kForm{\lambda_0}{k}+\kForm{\lambda_\perp}{k}$ with $\kForm{\lambda_0}{k} \in \mathcal{Z}_h^{(k)}$ and $\kForm{\lambda_\perp}{k} \in \mathcal{Z}_h^{(k),\perp}$. The boundedness of the projection follows from two estimates:
    \begin{enumerate}
        \item \textbf{$L^2$-Boundedness in $\mathcal{Z}_h^{(k)}$}
        Take $\kForm{\eta}{k} = \kForm{\lambda_0}{k}$ in \eqref{eq:BCP}, then
        \[ \normLtwo{ \kForm{\lambda_0}{k} }{k}^2 = \left ( \kForm{\lambda_0}{k} , \kForm{\phi}{k}\right ) \leq \normLtwo{ \kForm{\lambda_0}{k} }{k} \normLtwo{  \kForm{\phi}{k} }{k} \quad \Longrightarrow \quad  \normLtwo{  \kForm{\lambda_0}{k} }{k} \leq \normLtwo{  \kForm{\phi}{k} }{k} \;. \]
        
        \item \textbf{$L^2$-Boundedness in $\mathcal{Z}_h^{(k),\perp}$}
        Take $\kForm{\eta}{k + 1} = \mathrm{d} \kForm{\lambda_{\perp}}{k}$ in \eqref{eq:BCP} gives
        \[
            \normLtwo{\mathrm{d}\kForm{\lambda_{\perp}}{k} }{k}^2 = \left (\mathrm{d}\kForm{\lambda_{\perp}}{k}, \mathrm{d}\kForm{\phi}{k} \right )
            \leq \normLtwo{\mathrm{d}\kForm{\lambda_{\perp}}{k} }{k} \normLtwo{\mathrm{d}\kForm{\phi}{k} }{k} \:\:\: \Longrightarrow \:\:\: \normLtwo{\mathrm{d}\kForm{\lambda_{\perp}}{k}}{k} \leq \normLtwo{\mathrm{d}\kForm{\phi}{k} }{k} \;.
        \]
    \end{enumerate}
    Using Poicar\'{e}'s inequality \eqref{eq:Poincare}
    \[  \normLtwo{ \kForm{\lambda_{\perp}}{k} }{k} \leq C_P  \normLtwo{ \mathrm{d} \kForm{\lambda_{\perp}}{k} }{k} 
     \leq C_P\normLtwo{ \mathrm{d} \kForm{\phi}{k} }{k} \;.
    \]
    Combining these estimates gives
    \begin{eqnarray}
         \normHk{\kForm{\lambda}{k}}{k}^2 & = & \normLtwo{ \kForm{\lambda}{k} }{k}^2 + \normLtwo{ \mathrm{d} \kForm{\lambda}{k} }{k}^2 \\
         & = & \normLtwo{ \kForm{\lambda_0}{k} }{k}^2 + \normLtwo{ \kForm{\lambda_{\perp}}{k} }{k}^2 + \normLtwo{ \mathrm{d} \kForm{\lambda_{\perp}}{k} }{k}^2 \nonumber \\
        & \leq & \normLtwo{ \kForm{\phi}{k} }{k}^2 + \left ( 1+C_P \right ) \normLtwo{ \mathrm{d} \kForm{\phi}{k} }{k}^2 \\
        & \leq & \left ( 1+C_P \right ) \left \{ \normLtwo{ \kForm{\phi}{k} }{k}^2 +  \normLtwo{ \mathrm{d} \kForm{\phi}{k} }{k}^2 \right \} \;, \nonumber
    \end{eqnarray}
    gives the bound for the projection
    \begin{equation}
        \normHk{ \kForm{\Pi}{k} \kForm{\phi}{k}}{k} \leq \sqrt{ 1+C_P} \normHk{ \kForm{\phi}{k}}{k} \;. \label{eq:sob_bound}
    \end{equation}
\end{proof}

Having established that
\(
\Pi^{(k)}
:
H\Lambda^{(k)}(\Omega)
\rightarrow
H\Lambda_h^{(k)}(\Omega)
\)
defines a bounded linear projector, we now examine the role of the harmonic variable
\(
\kForm{h}{k+1}
\in
\mathcal{H}_h^{(k+1)}
\)
appearing in \eqref{eq:BCP}. This variable acts as a Lagrange multiplier enforcing the constraint
\(
\kForm{\lambda}{k+1}
\in
\mathcal{R}_h^{(k+1)}.
\)
Moreover, it admits a natural geometric interpretation.

\begin{lemma}[Harmonic multiplier]\label{lem:harm}
The harmonic form
\(
\kForm{h}{k+1}
\in
\mathcal{H}_h^{(k+1)}
\)
defined by \eqref{eq:BCP} is the $L^2$-orthogonal projection of the exact form
\(
\mathrm d \kForm{\phi}{k}
\)
onto the discrete harmonic space
\(
\mathcal H_h^{(k+1)}
\).
\end{lemma}

\begin{proof}
Choose
\(
\kForm{\eta}{k+1}
=
\kForm{\bar h}{k+1}
\in
\mathcal H_h^{(k+1)}
\)
in the second equation of \eqref{eq:BCP}. Since
\(
\mathrm d \kForm{\bar h}{k+1}=0
\),
and \(\bilnear{\kForm{\bar{h}}{k + 1}}{\mathrm{d}\kForm{\lambda}{k}} = 0\) the first two terms vanish, yielding
\[
\bilnear{\kForm{\bar h}{k+1}}
         {\kForm{h}{k+1}}
=
\bilnear{\kForm{\bar h}{k+1}}
         {\mathrm d\kForm{\phi}{k}},
\qquad
\forall
\kForm{\bar h}{k+1}
\in
\mathcal H_h^{(k+1)}\,
\]
and we have that
\[ \normHk{h^{(k+1)}}{k+1} \leq \normHk{\mathrm{d}\phi^{(k)}}{k+1} \;.\]
\end{proof}
At the infinite-dimensional and finite-dimensional level, exact forms are orthogonal to harmonic forms,
\[
\mathcal R^{(k+1)}
\perp
\mathcal H^{(k+1)}, \qquad  \mathcal R_h^{(k+1)}
\perp
\mathcal H_h^{(k+1)}.
\]
Consequently, if the finite-dimensional space is topologically conforming, i.e.
\(
\mathrm{gap}
\bigl(
\mathcal H^{(k+1)},
\mathcal H_h^{(k+1)}
\bigr)
=
0,
\)
then the orthogonal projection of the exact form
\(
\mathrm d\kForm{\phi}{k}
\)
onto
\(
\mathcal H_h^{(k+1)}
\)
vanishes, yielding
\(
\kForm{h}{k+1}
=
0.
\)
However, when the finite-dimensional space is not topologically conforming, the discrete harmonic space may contain components that are not orthogonal to exact forms. In that case,
\(
\kForm{h}{k+1}
\neq 0,
\)
and the harmonic multiplier measures the topological inconsistency between the infinite and finite-dimensional complexes. Naturally, $\kForm{h}{k + 1} = 0$ is trivial on contractible domains, for which
\(
\mathcal H^{(k+1)}=\{0\}
\).

\begin{lemma}[Commuting property]\label{lem:commute}
    Assume either (i) $\Omega$ is contractible, or (ii) $\mathrm{gap}\bigl(\mathcal H^{(k)},\mathcal H_h^{(k)}\bigr)=0$, i.e. the finite-dimensional sub-complex is topologically conforming. Then the projector commutes with the exterior derivative
\[
\mathrm d
\Pi^{(k)}
=
\Pi^{(k+1)}
\mathrm d.
\]
\end{lemma}
\begin{proof}
    We first consider the second case of a non-contractable domain with a topologically conforming sub-complex. If we were to insert $\mathrm{d}\kForm{\phi}{k}$ in \eqref{eq:BCP} instead of $\kForm{\phi}{k}$, we obtain
    \begin{equation}
    \left \{ \begin{array}{ccccccccc}
    \bilnear{\kForm{\eta}{k+1}}{\kForm{\underline{\lambda}}{k+1}} & + & \bilnear{\mathrm{d} \kForm{\eta}{k+1}}{\kForm{\underline{\lambda}}{k + 2}} & & & & & = & \bilnear{\kForm{\eta}{k+1}}{\mathrm{d}\kForm{\phi}{k}}  \\ [1.5ex]
    \bilnear{\kForm{\eta}{k + 2}}{\mathrm{d} \kForm{\underline{\lambda}}{k+1}} &  & + &  & \bilnear{\mathrm{d} \kForm{\eta}{k + 2}}{\kForm{\underline{\lambda}}{k + 3}}  & + &\bilnear{\kForm{\eta}{k + 2}}{\kForm{\underline{h}}{k + 2}} & = & \bilnear{\kForm{\eta}{k + 2}}{0} \\ [1.5ex]
      & & \bilnear{\kForm{\eta}{k + 3}}{\mathrm{d} \kForm{\underline{\lambda}}{k + 2}} & &  & & & = & \bilnear{\kForm{\eta}{k + 3}}{0} \\ [1.5ex]
     & & \bilnear{\kForm{\gamma}{k + 2}}{\kForm{\underline{\lambda}}{k + 2}} & & & & & = & \bilnear{\kForm{\gamma}{k + 2}}{0}
    \end{array} \right . \; .
\label{eq:BCP_dphi}
\end{equation}
In this equation, we underline the finite-dimensional forms on the left-hand side of \eqref{eq:BCP_dphi} to highlight that $\kForm{\lambda}{k+1}$ in \eqref{eq:BCP} is not necessarily the same as $\kForm{\underline{\lambda}}{k+1}$ in \eqref{eq:BCP_dphi}. We make a particular choice of $\kForm{\eta}{k+2}\in \mathcal{Z}_h^{(k+2)}$ in the second equation, in which case the term $\bilnear{\mathrm{d} \kForm{\eta}{k + 2}}{\kForm{\underline{\lambda}}{k + 3}}$ cancels out of this equation. As a consequence $\mathrm{d}\kForm{\underline{\lambda}}{k+1}=\kForm{\underline{h}}{k+2}=0$, because if $\kForm{\underline{h}}{k+2}$ were non-zero, we would have $-\mathrm{d}\kForm{\underline{\lambda}}{k+1} =\kForm{\underline{h}}{k+2}$ contradicting the definition of a harmonic form. So $\kForm{\underline{\lambda}}{k+1} = \mathrm{d}\kForm{\alpha}{k}$ for some $\kForm{\alpha}{k}\in H\Lambda_h^{(k)}(\Omega)$. If we insert this in the first equation of \eqref{eq:BCP_dphi} we have
\[ \bilnear{\kForm{\eta}{k+1}}{\mathrm{d}\kForm{\alpha}{k}}  +  \bilnear{\mathrm{d} \kForm{\eta}{k+1}}{\kForm{\underline{\lambda}}{k + 2}}    =  \bilnear{\kForm{\eta}{k+1}}{\mathrm{d}\kForm{\phi}{k}} \;. \]
If we subtract this equation from the second equation in \eqref{eq:BCP} we have
\[ \bilnear{\kForm{\eta}{k+1}}{\mathrm{d} ( \kForm{\alpha}{k} - \kForm{\lambda}{k})}  +  \bilnear{\mathrm{d} \kForm{\eta}{k+1}}{( \kForm{\underline{\lambda}}{k + 2} - \kForm{\lambda}{k+2} )}  -  \bilnear{\kForm{\eta}{k+1}}{ \kForm{h}{k+1} }   =  \bilnear{\kForm{\eta}{k+1}}{0} \;,\]
which needs to hold for all $\kForm{\eta}{k+1}\in H\Lambda_h^{(k+1)}(\Omega)$ and therefore also for $\kForm{\eta}{k+1}\in \mathcal{Z}_h^{(k+1)}$, in which case we have
\[ \bilnear{\kForm{\eta}{k+1}}{\mathrm{d} ( \kForm{\alpha}{k} - \kForm{\lambda}{k} )}      =  \bilnear{\kForm{\eta}{k+1}}{\kForm{h}{k+1} }, \qquad \forall \kForm{\eta}{k+1}\in \mathcal{Z}_h^{(k+1)} \;.\]
Following Lemma~\ref{lem:harm}, we know that $\kForm{h}{k+1} = 0$ if $\mathrm{gap}\bigl(\mathcal H^{(k)},\mathcal H_h^{(k)}\bigr)=0$ which thus yields $\mathrm{d}\kForm{\lambda}{k} = \mathrm{d}\kForm{\alpha}{k} = \kForm{\underline{\lambda}}{k+1} \Rightarrow{} \mathrm{d} \kForm{\Pi}{k} \kForm{\phi}{k} = \kForm{\Pi}{k + 1}\mathrm{d}\kForm{\phi}{k}$ and proves the commuting property. 

For the case of a contractable domain, $\kForm{h}{k+1} = \kForm{h}{k+2} = 0$ is trivially satisfied and the commuting property is achieved.
\end{proof}

\section{Approximate bounded cochain projection}\label{sec:pro_proj}

In this final section, we propose the construction of a approximate bounded cochain projection based on a two-level discrete de Rham structure. 

\begin{theorem}[Approximate bounded cochain projector on non-contractible domains]
    Let $\Omega$ be a non-contractible $n$-manifold with sufficiently smooth Lipschitz boundary. Let
    \(
    H\Lambda_h^{(k)}(\Omega)
    \subset
    H\widetilde{\Lambda}_h^{(k)}(\Omega)
    \subset
    H\Lambda^{(k)}(\Omega),
    \)
    be conforming finite-dimensional subspaces forming de~Rham subcomplexes. Assume that the enriched complex $H\widetilde{\Lambda}_h^{(k)}(\Omega)$ is topologically conforming, i.e.
    \(
    \mathrm{gap}
    \bigl(
    \mathcal H^{(k)},
    \widetilde{\mathcal H}_h^{(k)}
    \bigr)
    =
    0.
    \)
    Let
    \[
    \Pi^{(k)}
    :
    H\Lambda^{(k)}(\Omega)
    \rightarrow
    H\widetilde{\Lambda}_h^{(k)}(\Omega)
    \]
    denote the bounded projector through the augmented system \eqref{eq:BCP}, and let
    \[
    \Pi_{\mathscr C}^{(k)}
    :
    H\widetilde{\Lambda}_h^{(k)}(\Omega)
    \rightarrow
    H\Lambda_h^{(k)}(\Omega)
    \]
    be the canonical projector. Then the composition
    \[
    \Pi_{\mathscr C}^{(k)}\Pi^{(k)}
    :
    H\Lambda^{(k)}(\Omega)
    \rightarrow
    H\Lambda_h^{(k)}(\Omega)
    \]
    defines a bounded cochain projector where the following diagram commutes
    \[
    \begin{tikzcd}
        \hdots \arrow[r] &
        H\Lambda^{(k-1)}(\Omega)
        \arrow[r,"\mathrm d"]
        \arrow[d,"\Pi^{(k-1)}"]
        &
        H\Lambda^{(k)}(\Omega)
        \arrow[r,"\mathrm d"]
        \arrow[d,"\Pi^{(k)}"]
        &
        H\Lambda^{(k+1)}(\Omega)
        \arrow[r]
        \arrow[d,"\Pi^{(k+1)}"]
        &
        \hdots
        \\
        \hdots \arrow[r] &
        H\widetilde{\Lambda}_h^{(k-1)}(\Omega)
        \arrow[r,"\mathrm d"]
        \arrow[d,"\Pi_{\mathscr C}^{(k-1)}"]
        &
        H\widetilde{\Lambda}_h^{(k)}(\Omega)
        \arrow[r,"\mathrm d"]
        \arrow[d,"\Pi_{\mathscr C}^{(k)}"]
        &
        H\widetilde{\Lambda}_h^{(k+1)}(\Omega)
        \arrow[r]
        \arrow[d,"\Pi_{\mathscr C}^{(k+1)}"]
        &
        \hdots
        \\
        \hdots \arrow[r] &
        H\Lambda_h^{(k-1)}(\Omega)
        \arrow[r,"\mathrm d"]
        &
        H\Lambda_h^{(k)}(\Omega)
        \arrow[r,"\mathrm d"]
        &
        H\Lambda_h^{(k+1)}(\Omega)
        \arrow[r]
        &
        \hdots
    \end{tikzcd}
    \]
\end{theorem}
\begin{proof}
    The projector $\Pi^{(k)}$ is idempotent, bounded, and commutes with the exterior derivative, as per Lemmas~\ref{lem:idem}, \ref{lem:bound}, \ref{lem:commute}. The composition $\Pi_{\mathscr C}^{(k)}\Pi^{(k)}$ is idempotent as we have
    \[
    \Pi_{\mathscr C}^{(k)}\Pi^{(k)} (\Pi_{\mathscr C}^{(k)}\Pi^{(k)}) = \Pi_{\mathscr C}^{(k)} (\Pi_{\mathscr C}^{(k)}\Pi^{(k)}) = \Pi_{\mathscr C}^{(k)}\Pi^{(k)},
    \]
    since both $\Pi_{\mathscr C}^{(k)}$ and $\Pi^{(k)}$ are idempotent. Furthermore, the two projectors individually satisfy the commuting property, and hence we have
    \[
    \mathrm{d}\Pi_{\mathscr C}^{(k)}\Pi^{(k)} = \Pi_{\mathscr C}^{(k)}\mathrm{d}\Pi^{(k)} = \Pi_{\mathscr C}^{(k)}\Pi^{(k)}\mathrm{d}.
    \]
    The boundedness of $\Pi_{\mathscr C}^{(k)}\Pi^{(k)}$ follows from the boundedness $\Pi^{(k)}$ and the fact that $\Pi_{\mathscr C}^{(k)}$ forms a linear map between two finite-dimensional subspaces of the same discrete de Rham structure. Since the intermediate space $H\widetilde{\Lambda}_h^{(k)}(\Omega)$ is finite-dimensional, every linear operator defined on this space is continuous \cite{Boffi2013MixedApplications}. Consequently, the canonical projector \(\Pi_{\mathscr C}^{(k)}\) has a finite operator norm
    \[
    \|\Pi_{\mathscr C}^{(k)}\|
    :=
    \sup_{\substack{\kForm{\widetilde{\omega}_h}{k} \in H\widetilde{\Lambda}_h^{(k)}(\Omega) \\ \kForm{\widetilde{\omega}_h}{k}\neq0}}
    \frac{
    \|\Pi_{\mathscr C}^{(k)}\kForm{\widetilde{\omega}_h}{k}\|_{H\Lambda^{(k)}(\Omega)}}
    {\|\kForm{\widetilde{\omega}_h}{k}\|_{H\Lambda^{(k)}(\Omega)}}
    <\infty.
    \]
    Moreover, the space $H\widetilde{\Lambda}_h^{(k)}(\Omega)$ is constructed solely to resolve the topology of the domain and is therefore independent of the computational mesh for $H\Lambda_h^{(k)}(\Omega)$. Hence, we have
    \[
    \|
    \Pi_{\mathscr C}^{(k)}
    {\Pi}^{(k)}
    \kForm{\phi}{k}
    \|_{H\Lambda^{(k)}(\Omega)} 
    \leq
    \|\Pi_{\mathscr C}^{(k)}\|\,
    \|
    {\Pi}^{(k)}
    \kForm{\phi}{k}
    \|_{H\Lambda^{(k)}(\Omega)}
    \stackrel{\eqref{eq:sob_bound}}{\leq}
    \|\Pi_{\mathscr C}^{(k)}\|\,
    \sqrt{1 + C_P}\,
    \|\kForm{\phi}{k}\|_{H\Lambda^{(k)}(\Omega)},
    \]
where $\|\Pi_{\mathscr C}^{(k)}\|\,\sqrt{1 + C_P}$ is independent of the computational mesh parameter.
    
    Therefore, the composition of the two operators satisfies the properties of a bounded cochain projector.
\end{proof}

\begin{remark}
    The space $H\widetilde{\Lambda}_h^{(k)}(\Omega)$ is independent of the particular solution $\kForm{\phi}{k}$, to be approximated and depends only on the topology of $\Omega$. One practical way of constructing this space would be to start with the initial space $H\Lambda_h^{(k)}(\Omega)$, apply $\Pi_{\mathscr{C}}^{(k)}\Pi^{(k)}$ and successively $p$-refine within each spectral element until the commuting error, $\normHk{\kForm{{h}}{k + 1}}{k + 1}$, drops below a prescribed tolerance (e.g near machine precision $\mathcal{O}(10^{-12})-\mathcal{O}(10^{-16})$). 
\end{remark}

\bibliographystyle{siamplain}
\bibliography{references}

\end{document}